\documentclass{article}
\usepackage[utf8]{inputenc}
\usepackage[T2A]{fontenc}
\usepackage{babel}

\usepackage{todonotes}

\usepackage[a4paper, margin=1.5cm]{geometry}

\usepackage{amsmath,amssymb,amsthm,amsfonts}
\usepackage{mathrsfs}
\usepackage{mathtools}
\mathtoolsset{showonlyrefs}
\usepackage{hyperref}
\hypersetup{
    colorlinks=true,
    linkcolor=red,
    filecolor=magenta,      
    urlcolor=cyan,
}

\usepackage{microtype}

\newcommand{\N}{{\mathbb N}}

\newcommand{\ch}{\operatorname{ch}}

\renewcommand{\phi}{\varphi}

\newtheorem{lemma}{Lemma}

\newtheorem{theorem}{Theorem}

\newtheorem{corollary}{Corollary}

\def\fs{\kern 0.5em}

\newcounter{problemcnt}
	
\newcounter{pun}[problemcnt]

\usepackage{graphicx}
\graphicspath{ {./Images/} }

\usepackage{appendix}

\title{On list chromatic numbers of 2-colorable hypergraphs}

\author{ Danila Cherkashin$^\mathrm{a,b}$,~
 Alexey Gordeev$^{\mathrm{c,d}}$\\
{\small ~a. Chebyshev Laboratory, St. Petersburg State University, 14th Line V.O., 29B, Saint Petersburg 199178 Russia}\\
{\small ~b. Moscow Institute of Physics and Technology, Laboratory of Combinatorial and Geometric Structures,}\\ 
{\small ~c. St.~Petersburg Department of V.~A.~Steklov Institute of Mathematics of the Russian Academy of Sciences}\\
{\small ~d. The Euler International Mathematical Institute, St. Petersburg, Russia}
\date{}
}

\begin{document}

\maketitle

\begin{abstract}
We give an upper bound on the list chromatic number of a 2-colorable hypergraph which generalizes the bound of Schauz on $k$-partite $k$-uniform hypergraphs. 
It makes sense for sparse hypergraphs: in particular we show that a $k$-uniform $k$-regular hypergraph has the list chromatic number 2 for $k \geq 4$.
Also we obtain both lower and upper bound on the list chromatic number of a complete $s$-uniform 2-colorable hypergraph in the vein of Erd{\H o}s--Rubin--Taylor theorem.
\end{abstract}

\section{Introduction}

A hypergraph $H$ is a pair of sets $(V, E)$, where $V$ is finite and $E \subset 2^V$; $V$ is called the set of vertices and $E$ the set of edges.
A hypergraph is said to be $n$-uniform if all of its edges have size $n$ (further we call them $n$-graphs). 
Note that a graph is a 2-uniform hypergraph. 

Given a hypergraph $H = (V, E)$ and a vertex $v \in V$, define the degree of the vertex $d_H(v)$ as the number of edges of $H$ that contain $v$.

A coloring of a hypergraph with $r$ colors is a map $f : V \to \{1, \dots, r\}$. A coloring $f$ of a hypergraph is said to be proper if each edge 
$e \in E$ contains two vertices $v_1, v_2 \in e$ such that $f(v_1) \neq  f(v_2)$.
The minimal number $r$ for which there exists a proper $r$-coloring of $H$ is called the chromatic number $\chi(H)$ of the hypergraph $H$.

%Hypergraph $H=(V, E)$ is bipartite if there exists a proper 2-coloring of $H$. 
Consider a 2-colorable hypergraph $H$. Its proper 2-coloring gives a partition of $V$ into two sets $A \cup B = V$, $A \cap B = \emptyset$ such that every edge of $H$ intersects both $A$ and $B$. Further we call such sets $A$ and $B$ parts of $H$. For a subset of a part $T \subset A$ define its neighborhood
\[
N_H(T) = \bigcup_{\substack{e \in E,\\e \cap T \neq \emptyset}} e \setminus A.
\]
For convenience define also $N_H(v) = N_H(\{v\})$.

Define a complete 2-colorable hypergraph $K^s_{n,m}$ as $s$-graph with sizes of parts $n$, $m$ and all possible edges of size $s$ intersecting both parts.

An orientation of a hypergraph $H=(V, E)$ is a map $\phi : E \to V$ which satisfies $\phi(e) \in e$ for any $e \in E$. For orientation $\phi$ define a degree function $d_\phi$ on the set of vertices as follows: $d_\phi(v) = |\phi^{-1}(v)|$.

To each vertex $v$ we assign a list $L(v)$ of colors which
can be used for $v$. Given a map $f : V \to \N$, we say that a hypergraph $H$ is $f$-list-colorable ($f$-choosable) if for any set of lists with lengths $|L(v)| = f(v)$ there exists a proper coloring of $H$. Hypergraph is $k$-choosable if it is $f$-choosable, where $f(v)=k$ for each $v$. The list chromatic number $\ch(H)$ of the hypergraph $H$ is the
minimum number $k$ such that $H$ is $k$-choosable.

List colorings of graphs and hypergraphs
were independently introduced by Vizing and by Erd{\H o}s, Rubin, and Taylor.
Clearly, $\ch(H) \geq \chi(H)$, since all lists can be taken equal to $\{1, \dots , \ch(H)\}$. 
At the same time, the chromatic number and the list chromatic number are different,
for example, for the complete bipartite graph $K_{3,3}$, for which $\ch(K_{3,3}) = 3$ (equality is attained at
the lists $\{1, 2\}, \{1, 3\}, \{2, 3\}$ assigned to the vertices of each part).

\paragraph{Sparse case.} For a hypergraph $H=(V,E)$ denote
\[
L(H)=\max_{\emptyset \neq E'\subset E}\frac{|E'|}{|\cup_{e \in E'} e|}.
\]

One of the methods of estimating the list chromatic number of a graph is the Alon--Tarsi method \cite{Alon1992}: if there exists an orientation $\phi$ of a graph $G$ with certain structural properties, then $G$ is $(d_\phi+1)$-choosable. In the case of bipartite (2-colorable) graphs any orientation has the desired properties, which gives two following theorems.

\begin{theorem}\label{thm:AlonTarsi}
If a bipartite graph $G=(V, E)$ has an orientation $\phi$, then $G$ is $(d_\phi+1)$-choosable.
\end{theorem}

\begin{theorem}[Theorem 3.2 in \cite{Alon1992}]
For any bipartite graph $G$, $\ch(G) \leq \lceil L(G) \rceil + 1$.
\end{theorem}

In 2010 Schauz~\cite{schauz2010paintability} generalized these results to $k$-partite $k$-uniform hypergraphs, i.e. $k$-uniform hypergraphs whose set of vertices can be partitioned into $k$ sets $V_1,\dots,V_k$ such that each edge has exactly one vertex in each of $V_i$ (bipartite graphs correspond to the case $k=2$). We show that in fact the same results hold in even more general case of $2$-colorable hypergraphs.
%<<This is so since each edge e of H has exactly one vertex in each partition class Vi of H>>~\cite{schauz2010paintability}.

\paragraph{Dense case.} Let $N(n, r)$ denote the minimum number of vertices in an $n$-partite ($n$-colorable) graph with the list chromatic number larger than $r$.
The following classic theorem express the asymptotics of $N(2,r)$ in terms of the minimal number of edges in an $n$-uniform hypergraph without a proper 2-coloring.
The latter quantity is denoted by $m(n)$.

\begin{theorem}[Erd{\H o}s -- Rubin -- Taylor~\cite{erdos1979some}]\label{thm:ERT}
For any $r$
\[
m(r) \leq N(2, r) \leq 2m(r).
\]
\end{theorem}
The problem of finding $m(n)$ is well-known, the best current bounds are
\[
c\sqrt{\frac{n}{\ln n}} 2^n \leq m(n) \leq (1+o(1)) \frac{e\ln 2}{4} n^22^n.
\]
For details see survey~\cite{raigorodskii2020extremal}.

Kostochka extended Theorem~\ref{thm:ERT} in two ways. Let $Q(n, r)$ be the minimal number of edges in an $n$-uniform $n$-partite hypergraph with the chromatic number greater than $r$;
$m(n,r)$ be the minimal number of edges in an $n$-uniform hypergraph with the chromatic number greater than $r$, and finally $p(n,r)$ be 
the minimal number of edges in an $n$-uniform hypergraph without an $r$-coloring such that every edge meets every color.

\begin{theorem}[Kostochka~\cite{kostochka2002theorem}] For every $n,r \geq 2$ the following inequalities hold
\[
m(r, n) \leq Q(n, r) \leq nm(r, n);
\]
\[
p(r, n) \leq N(n, r) \leq np(r, n).
\]
\end{theorem}

Upper and lower bounds on the quantities $m(n,r)$ and $p(n,r)$ heavily depend on the relations between $n$ and $r$. The picture is collected in survey~\cite{raigorodskii2020extremal} (see also more recent paper~\cite{akhmejanova2020chain}).

\section{Sparse case}

We give an alternative, more direct proof of the next theorem in the Appendix, since it is rather concise.

\begin{theorem}\label{thm:orientation}
If a 2-colorable hypergraph $H = ( V, E )$ has an orientation $\phi$, then $H$ is $(d_\phi+1)$-choosable.
\end{theorem}
\begin{proof}
Since $H$ is 2-colorable, we can choose a vertex $u_e\in e$ in each edge $e$ such that $\phi(e)$ and $u_e$ belong to different parts of $H$.

Consider a bipartite graph $B$ on the set of vertices $V$ with the set of edges $\cup_{e\in E} \{ (\phi(e),u_e) \}$.
Consider the following orientation of $B$: $\phi'( (\phi(e),u_e) ) = \phi(e)$. Note that $d_{\phi'}(v) = d_\phi(v)$ for any $v \in V$. Then, by Theorem~\ref{thm:AlonTarsi}, $B$ is $(d_\phi+1)$-choosable.
Since any proper coloring of $B$ is a proper coloring of $H$, it follows that $H$ is $(d_\phi + 1)$-choosable.
\end{proof}

The following lemma first appeared in \cite{schauz2010paintability} (see Lemma 3.2). We give a proof of it here for the sake of completeness.

\begin{lemma}\label{lm:Hall}
For any hypergraph $H = (V, E) $ there exists an orientation $\phi$ of $H$ with $d_\phi \leq \lceil L(H) \rceil$.
\end{lemma}
\begin{proof}
Consider a bipartite graph $G$ with parts $V \times \{1,\dots,\lceil L(H) \rceil\}$ and $E$, in which $(v,i)$ and $e$ are connected if and only if $v$ and $e$ are incident in $H$.
The statement of lemma is true if and only if there is a matching in $G$ which covers $E$.

Let $\emptyset \neq T\subset E$ be an arbitrary subset of edges of $H$. We estimate the size of the neighborhood of $T$ in $G$:
\[
|N_G(T)| = |\cup_{e \in T} e| \cdot \lceil L(H) \rceil \geq |\cup_{e \in T} e| \cdot \frac{|T|}{|\cup_{e \in T} e|} = |T|.
\]
By Hall's marriage theorem it follows that there exists a matching in $G$ which covers $E$.
\end{proof}

Theorem~\ref{thm:orientation} and Lemma~\ref{lm:Hall} combined give the following theorem.

\begin{theorem}\label{thm:sparse}
For any 2-colorable hypegraph $H$, $\ch(H) \leq \lceil L(H) \rceil + 1$.
\end{theorem}

The exact value of $L(H)$ can be difficult to calculate, so we provide a bound in simpler terms. For a hypergraph $H$ denote the maximum degree of a vertex in $H$ by $\Delta(H)$ and the minimum size of an edge in $H$ by $s(H)$.

\begin{corollary}
Let $H = ( V, E )$ be a 2-colorable hypergraph. Then
\[
\ch(H) \leq \left\lceil \frac{\Delta(H)}{s(H)} \right\rceil + 1.
\]
\end{corollary}
\begin{proof}
For any $E' \subset E$ we have
\[
|E'| \cdot s(H) \leq \sum_{e \in E'} |e| \leq |\cup_{e \in E'} e| \cdot \Delta(H),
\]
so 
\[
L(H) \leq \frac{\Delta(H)}{s(H)}.
\]
\end{proof}

For an arbitrary (not necessarily 2-colorable) hypergraph a weaker bound with an additional factor of 2 holds. The next theorem was essentially proved by Gravin and Karpov~\cite{gravinkarpov}, though they were not considering their results in the context of list colorings. 
In fact, the theorem in~\cite{gravinkarpov} has the second part in which <<+1>> is removed under some assumptions; it can be also generalized on the list chromatic numbers.

\begin{theorem}
Let $H = ( V, E )$ be a hypergraph. Then
\[
\ch(H) \leq \left\lceil \frac{2\Delta(H)}{s(H)} \right\rceil + 1.
\]
\end{theorem}
\begin{proof}
Denote
\[
k=\left\lceil\frac{2\Delta(H)}{s(H)}\right\rceil.
\]
Let $G$ be the incidence graph of $H$, i.e. a bipartite graph with parts $V$, $E$ and with $v \in V$, $e \in E$ connected in $G$ if and only if $v$ and $e$ are incident in $H$. We want to delete some edges from $G$ to obtain a graph $K$ with the following conditions on degrees:
\[
d_K(e)=2 \text{ for any } e\in E;\quad \max_{v\in V} d_K(v)\leq k.
\]

If there exists such $K$, denote $N_K(e) = \{v_e, u_e\}$ and consider a graph $B$ on a set of vertices $V$ with the set of edges $\cup_{e\in E} \{ (v_e,u_e) \}$. One can color vertices of $B$ in any order, each time using any color not yet taken by adjacent vertices, to obtain a bound $\ch(B) \leq \Delta(B) + 1 \leq k + 1$.
Since any proper coloring of $B$ is a proper coloring of $H$, it follows that $\ch(H)\leq k + 1$.

If there is no $K$ with desired properties, then we choose such $K = (V, E_K)$ that $d_K(e)=2$ for any $e \in E$ and $p(K)=\sum_{v\in V}\max(0,d_K(v)-k)$ is as small as possible. Denote by $S\subset V$ the set of vertices with $d_K>k$. Define an \textit{augmenting path} as such a sequence of vertices and edges $v_1,e_1,\dots,v_m,e_m,v_{m+1}$ that $v_i\in V$, $e_i\in E$, $(v_i,e_i)\in E_K$, $(v_{i+1},e_i)\not\in E_K$, $e_i$ are pairwise distinct, $v_1 \in S$. Let $U\subset V$ be the subset of vertices of $V$ reachable by an augmenting path. Note that:
\begin{itemize}
    \item For any $v\in U$ the inequality $d_K(v)\geq k$ holds (otherwise we could flip the status of edges on the corresponding augmenting path and decrease $p(K)$).
    \item For any $e\in E$ if $N_K(e)\cap U\neq\emptyset$ then $|N_G(e)\setminus U|\leqslant 1$. Indeed, suppose that $v\in N_K(e)\cap U$, then any $w\in N_G(e)\setminus N_K(e)$ must lie in $U$, because there is an augmenting path ending in $w$; it follows that the only vertex of $N_G(e)$ which can possibly not lie in $U$ is the unique vertex from $N_K(e)\setminus\{v\}$.  
\end{itemize}

Let $T\subset N_K(U)$ be a set of $e\in N_K(U)$ such that $|N_G(e)\setminus U|=1$. Since $d_K(v)\geqslant k$ for any $v\in U$ and $d_K(v)>k$ for any $v\in S$, we get the inequality
\[
|N_K(U)\setminus T|>\frac{k|U|-|T|}{2}.
\]
Now we obtain the lower bound on the sum of degrees $d_G$ over vertices of $U$ by summing up $|N_G(e)\cap U|$ over $e\in N_K(U)$:
\[
\sum_{v\in U}d_G(v)>|T|(s(H)-1)+\frac{k|U|-|T|}{2}s(H)\geq \Delta(H)|U|+|T|\left(s(H)-1-\frac{s(H)}{2}\right)\geq \Delta(H)|U|.
\]
It follows that there exists such $v\in U$ with $d_G(v)>\Delta(H)$, which is a contradiction.
\end{proof}

\section{Dense case}

\begin{theorem}
\label{lowerERT}
Let $H$ be an $s$-uniform 2-colorable hypergraph on $t$ vertices and
\[
t < \frac{1}{4} (1 + s^{1/l})^{l}.
\]
Then 
\[
\ch (H) \leq l.
\]
\end{theorem}

\begin{proof}

Consider the set of lists as a hypergraph $F$. Then $t = |E(F)|$; we still refer to edges of $F$ as lists to avoid the confusion with edges of $H$. 
We split the palette $V(F)$ randomly into three parts: blue colors, red colors and neutral colors in the following way.
Every vertex of $V(F)$ is red or blue uniformly and independently with probability $\frac{1-p}{2}$, so with probability $p$ it is neutral; the value of $p$ will be defined later. 
Then the expectation of monochromatic lists is equal to
\[
A = 2\left (\frac{1-p}{2} \right)^l |E(F)|.  
\]
We call a list \textit{dangerous}, if it has no blue or no red vertices.
The expectation of dangerous lists equals to
\[
B = 2\left (\frac{1+p}{2} \right)^l |E(F)|.  
\]
In the case $A < 1/2$ and $B < s/2$ Markov inequality implies that with positive probability $F$ has no monochromatic edges and the number of dangerous lists is smaller than $s$. 
Recall that $H$ is 2-colorable; under the assumption one can color vertices of one part of $H$ in blue or neutral colors, vertices of another part of $H$ in red or neutral colors, in such a way that vertices of $H$ corresponding to dangerous lists have neutral colors. 
Since we have less than $s$ dangerous lists, and every other edge of $H$ has a red and a blue vertex, we construct a proper coloring of $H$.

Now we show that the following choice of $p$ satisfies the desired inequalities. Put
\[
p = \frac{s^{1/l} - 1}{1 + s^{1/l}}, \quad \mbox{ then } \quad \frac{B}{A} = \left (\frac{1+p}{1-p}\right)^l = s.
\]
On the other hand,
\[
\left (\frac{2}{1-p} \right)^l = \left (1 + s^{1/l} \right)^l.
\]
Summing up, 
\[
A = \frac{2|E(F)|}{\left (1 + s^{1/l} \right)^l} = \frac{2t}{\left (1 + s^{1/l} \right)^l} < \frac{1}{2},
\]
hence
\[
B = A \left (\frac{1+p}{1-p} \right)^l = As < \frac{s}{2}, 
\]
so the condition on $t$ implies $\ch(H) \leq l$.
\end{proof}

Obviously, $\ch (K_{t/2,t/2}^s) \leq \ch (K_{t/2,t/2}^2) = (1+o(1)) \log_2 t$ by Erd{\H o}s--Rubin--Taylor theorem.
Note that for $s = t^\alpha$, where $\alpha$ is a constant, the bound in Theorem~\ref{lowerERT} is constant times better.

\begin{corollary}
Suppose that
\[
t \leq \sqrt{s} \cdot 2^{l-2}.
\]
Then
\[
\ch (K_{k,t-k}^s) \leq l
\]
for every $k$ from 1 to $t-1$.
\end{corollary}
\begin{proof}
Indeed,
\[
1 + s^{1/l} \geq 2s^{1/(2l)}.
\]
Hence
\[
(1 + s^{1/l})^l \geq \sqrt{s}2^l
\]
and we are done by Theorem~\ref{lowerERT}.
\end{proof}

\begin{theorem}
Suppose that
\[
t = \Omega \left( (\log s + \log l) \cdot l^2 (1 + s^{1/l})^l \right).
\]
Then
\[
\ch (K_{t/2,t/2}^s) > l.
\]
\end{theorem}

\begin{proof}
Denote $H = K_{t/2,t/2}^s$.

Put $v = l^2$ and consider a set of random (independent and uniform) lists with size $l$ over the left part of $H$, and its copy over the right part.
Let $F$ be the resulting random hypergraph of colors; we refer to its edges as lists to avoid the confusion with the edges of $H$. 

Suppose the contrary: in particular it means that for every $F$ hypergraph $H = K_{t/2,t/2}^s$ has proper coloring in colors of $F$. Consider an arbitrary such proper coloring $\pi_H$.
We call a color (a vertex of $F$) \textit{poor}, if it appears in both parts of $H$; so a poor color appears at most $s-1$ times.
Define a coloring $\pi_F$ on vertices of $F$ as follows: if a color appears only in the left part of $H$ then the corresponding vertex of $F$ is blue, if a color appears only in the right part of $H$ then the corresponding vertex of $F$ is red, finally if a color is poor then the corresponding vertex has no color.
Note that $\pi_F$ contains no monochromatic lists; indeed if a list is blue then one cannot choose a color from a copy of this list over the vertex in the right part of $H$, which contradicts with the existence of $\pi_H$.
We show that with positive probability such a coloring $\pi_F$ cannot exist.

Consider an arbitrary vertex coloring $\pi$ of $F$ in two colors in which some vertices remain colorless. 
We call a list \textit{dangerous} if it has no blue or no red vertices.
It turns out that with positive probability every $\pi$ has a monochromatic list or at least $v_0(s-1)$ dangerous lists, where $v_0$ is the number of colorless vertices. 
Hence $\pi_F$ (and also $\pi_H$) cannot exist, which gives the desired contradiction. Below we provide the calculations that the probability that !!!!!!!!

Let $c = \frac{v_0}{v}$ be the ratio of colorless vertices, denote by $p_1 = p_1(\pi)$ the probability that a random list is monochromatic and by $p_2 = p_2(\pi)$ the probability that it is dangerous. By definition
\[
p_1 = \left( \frac{1-c}{2} \right )^l, \quad \quad p_2 = \left( \frac{1 + c}{2} \right )^l.
\]

Then the probability that we have at most $(s-1)cv$ dangerous lists is smaller than the probability that there are at most $sl^2$ dangerous lists. The latter is evaluated by
\[
\sum_{i = 1}^{sl^2} \binom{t/2}{i} p_2^i (1-p_2)^{t/2 - i}  \leq \sum_{i=1}^{sl^2} \left( \frac{tp_2}{2} \right)^{i}   e^{-p_2(t/2 - i)} \leq \sum_{i=1}^{sl^2} \left( \frac{tp_2}{2} \right)^{i}   e^{-p_2(t/2 - sl^2)}.
\]
Consider the following substitution $c = \frac{s^{1/l}-1}{s^{1/l}+1}$ and put $T = \frac{t}{2(s^{1/l}+1)^l}$. We got
\[
p_2 = \left(\frac{1+c}{2}\right)^l = \left (\frac{s^{1/l}}{s^{1/l}+1} \right)^l = \frac{s}{(s^{1/l}+1)^l}.
\]
Then $\frac{tp_2}{2} = sT$ and 
\[
\sum_{i=1}^{sl^2} \left( \frac{tp_2}{2} \right)^{i}   e^{-p_2(t/2 - sl^2)} \leq {sl^2} \left( \frac{tp_2}{2} \right)^{sl^2}   e^{-p_2(t/2 - sl^2)}. 
\]
One has
\[
e^{-p_2(t/2 - sl^2)} \leq e^{-(p_2t/2-sl^2)} = e^{-sT + sl^2}.
\]
Also 
\[
sl^2 \left( \frac{tp_2}{2} \right)^{sl^2} = (sT + o(1))^{sl^2} = e^{(1+o(1)) \log(sT) sl^2}.
\]
Summing up, under the conditions of theorem
\[
(1+o(1)) l^2 s \log (sT) - sT + sl^2 < 0
\]
and the event that the number of dangerous lists is smaller than $sl^2$ has probability smaller than $\frac{1}{2}$.

On the other hand
\[
(1-p_1)^{t/2} \leq e^{-p_1t/2} \leq e^{-\frac{t}{2(1+s^{1/l})^l}} = e^{-T} < \frac{1}{2}.
\]
So for $c = \frac{s^{1/l}-1}{s^{1/l}+1}$ a random $l$-graph of lists with positive probability has both a monochromatic list and at least $sl^2$ dangerous lists in every 2-coloring. 
From a combinatorial argument for smaller $c$ a random $l$-graph of lists with positive probability has a monochromatic list in every 2-coloring, and for bigger $c$ 
with positive probability has at least $sl^2$ dangerous lists in every 2-coloring. 

\end{proof}

\section{Applications and discussion}

\begin{enumerate}
    \item A hypergraph is $k$-\textit{regular} if the degree of every vertex is equal to $k$.
    It is known that a $k$-uniform $k$-regular hypergraph is 2-colorable for $k \geq 4$~\cite{thomassen1992even,henning20132}. 
    Thus Theorem~\ref{thm:sparse} gives that such graph has the list chromatic number 2. So $k$-uniform $k$-regular hypergraphs are chromatic-choosable for $k \geq 4$.

    %\item Note that Продолжим понимать место этих оценок в известном мире. Например, если думать про $m(3,r)$, то оценка очень слабая.
%Ведь по факту тут получается условие $r = O( \frac{e}{v})$, а мы как-то умеем понимать, что все вершины степени меньше $r^2$ нас мало волнуют, а если рассматривать подграф на степенях $r^2$, то $r = O(\sqrt{e/v})$. Может быть случай, когда все ребра имеют размер хотя бы 3 нуждается В УСИЛЕНИИ? Щепоточка случайности иногда лечит такие проблемы. 

    %\item Не могу придумать примеры, когда теорема из секции три точна. Наверное всем кроме меня известно, когда там точна теорема Алона -- Тарси.
    
\item It turns out that there is a large gap between the bounds in dense and sparse cases; the same holds even for 2-graphs.
As far as we are aware, the best known general bounds on the choice number of $d$-regular bipartite graph $G$ are the following (see~\cite{alon2000degrees})
\[
\left(\frac{1}{2} - o(1)\right) \log d\leq \ch (G) \leq c \frac{d}{\log d}.
\]
Note that Erd{\H o}s -- Rubin -- Taylor gives a tight bound for a complete bipartite graph
\[
ch(K^2_{t,t}) = (1+o(1)) \log t. 
\]

\end{enumerate}

%\appendix
%\appendixpage

\paragraph{Acknowledgments.} The research of Danila Cherkashin is supported by «Native towns», a social investment program of
PJSC «Gazprom Neft». The research of Alexey Gordeev was funded by RFBR, project number 19-31-90081.

\bibliographystyle{plain}
\bibliography{main}

\section*{Appendix. Alternative proof of Theorem~\ref{thm:orientation}}

\begin{proof}
Denote parts of $H$ as $A,B\subset V$; every edge of $H$ intersects both $A$ and $B$. For each edge $e$ of $H$ choose a spanning tree $T_e$ on its set of vertices such that each edge of $T_e$ connects a vertex from part $A$ with a vertex from part $B$. Consider the following polynomial on variables $\{x_a\}_{a\in A}$, $\{y_b\}_{b\in B}$:
\[
F_H(x,y)=\prod_{e\in E}\left(\sum_{(a,b)\in T_e} (x_a-y_b)\right).
\]
Note that if $F_H(x,y)\neq 0$ then the values of $x,y$ form a proper coloring of $H$ (the reverse is not necessarily true). By Combinatorial Nullstellensatz (see \cite{Alon1999comb}), if the coefficient 
\[
\left[\prod_{a\in A}x_a^{d_\phi(a)}\prod_{b\in B}y_b^{d_\phi(b)}\right]F_H\neq 0,
\]
then $H$ is $(d_\phi+1)$-choosable. Consider now
\[
F^*_H(x,y)=F_H(x,-y)=\prod_{e\in E}\left(\sum_{(a,b)\in T_e} (x_a+y_b)\right).
\]
Note that
\[
\left[\prod_{a\in A}x_a^{d_\phi(a)}\prod_{b\in B}y_b^{d_\phi(b)}\right]F_H=\pm\left[\prod_{a\in A}x_a^{d_\phi(a)}\prod_{b\in B}y_b^{d_\phi(b)}\right]F^*_H,
\]
hence one can study coefficients of $F^*_H$ instead of $F_H$. The coefficient of $F^*_H$ in question is nonzero, because no summands will cancel out after opening the brackets, and orientation $\phi$ corresponds to at least one summand with the desired coefficient.
\end{proof}

\end{document}